\documentclass[12pt,reqno,a4paper]{amsart}
\usepackage{lmodern}
\usepackage{fullpage}

\usepackage{setspace}
\usepackage{datetime}
\usepackage{verbatim}
\usepackage{url}
\usepackage{dsfont}
\usepackage[dvipsnames]{xcolor}
\usepackage{tikz}
\usepackage{amsmath}
\usepackage{cite}
\usepackage{euflag}

\usepackage{graphicx,caption,subcaption,epstopdf}

\usepackage[shortlabels]{enumitem} 

\allowdisplaybreaks

\usepackage[mathlines]{lineno}
\usepackage{etoolbox} 

\newtheorem{theorem}             {Theorem}[section]
 
\newtheorem{lemma}     [theorem] {Lemma}      
\newtheorem{fact}     [theorem] {Fact}

\newtheorem{proposition}[theorem] {Proposition}   
\newtheorem{claim}[theorem] {Claim}

\newenvironment{claimproof}[1][Proof of claim]{\begin{proof}[#1]}{\end{proof}}

\newtheorem{proto-theo}[theorem] {Proto-Theorem}   

\def\tc{\mathop{\text{\rm tc}}\nolimits}
\def\tp{\mathop{\text{\rm tp}}\nolimits}

\newcommand{\blue}{\mathrm{blue}}
\newcommand{\red}{\mathrm{red}}
\newcommand{\mix}{\mathrm{mix}}

\newcommand{\out}{\mathrm{out}}
\newcommand{\R}{\mathrm{R}}
\newcommand{\B}{\mathrm{B}}

\usepackage[hidelinks]{hyperref}

\begin{document}
	\pagestyle{plain}
	\thispagestyle{empty}
	\footskip=30pt
	\shortdate
	\settimeformat{ampmtime}
	\onehalfspacing

	\title[Monochromatic components in dense $2$-edge-coloured balanced bipartite graphs]%
	{Monochromatic components in dense $2$-edge-coloured balanced bipartite graphs}
	
	\author[C.~Bispo]{César Bispo}
    \author[G.~Kontogeorgiou]{George Kontogeorgiou}
    \author[M.~Lage]{Marcelo Lage}
	\author[G.~O.~Mota]{Guilherme~O.~Mota}
	\address{Instituto de Matem\'atica e Estat\'{\i}stica \\Universidade de 
		S\~ao Paulo, Rua do Mat\~ao 1010\\05508--090~S\~ao Paulo, Brazil}
	\email{\{cesar.bispo,\,marcelomlage,\,mota\}@ime.usp.br}

    \author[B.~Skarmeta]{Bruno Skarmeta}
    \address{Centro de Modelamiento Matem\'atico, Beaucheff 851, Santiago, Chile}
    \email{gkontogeorgiou@dim.uchile.cl}
    \email{bruno.skarmeta@ug.uchile.cl}


	\begin{abstract}
	We prove that each $2$-edge-coloured spanning subgraph $G$ of $K_{n,n}$ with
	$\delta(G)\ge \lfloor (2n+1)/3 \rfloor$ can be covered by at most three
	monochromatic components. We provide a $2$-edge-coloured spanning subgraph of
	$K_{n,n}$ showing this minimum degree condition is sharp.
	\end{abstract}

	\maketitle
	\onehalfspacing

\section{Introduction}

Monochromatic covers and partitions in edge-coloured graphs are classical topics
in Ramsey theory, introduced in the early twentieth
century~\cite{ramsey1930problem, erdos1935combinatorial, erdos1947some}. Given a
graph~$G$ and a positive integer $r$, let $\tc_r(G)$ be the least integer $k$
such that, in every $r$-edge-colouring of~$G$, the set $V(G)$ can be covered by
at most $k$ monochromatic components. Similarly, let $\tp_r(G)$ be the least
integer $k$ such that, in every~$r$-edge-colouring of $G$, the set $V(G)$ can be
partitioned into at most $k$ monochromatic components.

For $r\ge2$, Erd\H{o}s, Gy\'arf\'as, and Pyber~\cite{erdos1991vertex}
conjectured that $\tp_r(K_n)\le r-1$. Considering bipartite graphs, a classical
conjecture of Gy\'arf\'as and Lehel asserts that, for $r\ge2$, every
$r$-edge-colouring of the complete bipartite graph $K_{n,m}$ can be covered by
at most $2r-2$ monochromatic connected components; that is, $\tc_r(K_{n,m})\le
2r-2$. This conjecture is closely related to Ryser-type covering problems for
hypergraphs; see~\cite{gyarfas1977partition, lehel1998ryser, chen2012around} for
the original formulation and further developments.

There are also several related problems in which one asks for partitions or
covers by more restrictive monochromatic structures, such as paths or cycles, as in the work of Gy\'arf\'as~\cite{gyarfas1983vertex}. DeBiasio and
Nelsen~\cite{debiasio2016monochromatic} considered monochromatic cycle
partitions under minimum-degree conditions, while Gy\'arf\'as, Ruszink\'o,
S\'ark{\"o}zy, and Szemer\'edi~\cite{gyarfas2006onesided} studied one-sided
coverings of coloured complete bipartite graphs. More recently, Benevides,
Quintino, and Talon~\cite{benevides2024partitioning} proved that every
$2$-edge-colouring of $K_{n,n}$ can be partitioned into at most four monochromatic
cycles, with vertices and edges viewed as degenerate cycles.

Recent work has focused on covering or partitioning the vertices of graphs
satisfying minimum-degree conditions into monochromatic components, beginning
with Bal and DeBiasio~\cite{bal2017partitioning}. For $2$-edge-colourings of general
graphs, Gir\~ao, Letzter, and Sahasrabudhe~\cite{girao2019partitioning}
strengthened earlier bounds by proving that every sufficiently large
$n$-vertex graph $G$ with minimum degree $\delta(G)\ge (2n-5)/3$ satisfies $\tp_2(G)\le 2$.

The analogous minimum-degree problem for bipartite graphs presents distinct
structural challenges. Recently, Fern\'andez, Pavez-Sign\'e, and
Stein~\cite{fernandez2024monochromatic} proved that, for every $\varepsilon>0$
and all sufficiently large $n$, every spanning subgraph $G$ of $K_{n,n}$ with
$\delta(G)\ge(13/16+\varepsilon)n$ satisfies $\tp_2(G)\le3$.

One might wonder whether it is possible to cover dense bipartite graphs using
only two monochromatic components, as in the non-bipartite setting. For
$n\ge2$, let $G$ be obtained from $K_{n,n}$, with bipartition $(X,Y)$, by
deleting the edges $x_1y$ and $x_2y$, where $x_1,x_2\in X$ are distinct and
$y\in Y$. Then $\delta(G)=n-2$. Colour red every edge incident with $y$ or~$x_1$, and colour blue every remaining edge. No monochromatic component
contains more than one of $x_1$, $x_2$, and $y$. Thus $\tc_2(G)\ge3$, showing
that any minimum-degree condition guaranteeing $\tc_2(G)\le2$ must require
minimum degree at least $n-1$. It is therefore impossible to obtain a linear improvement, with respect to the trivial bound $n$, on the minimum degree required for covering with two monochromatic components. In fact, even a random bipartite graph with density $p$ up to $1-3\frac{\log n}{n}$, with high probability, cannot be covered by two monochromatic components~\cite{fernandez2024monochromatic}.   

We therefore determine the minimum-degree threshold that guarantees $\tc_2(G)\le3$ for
every spanning subgraph $G$ of $K_{n,n}$. Our main result is the following.

\begin{theorem}\label{thm:main}   
Let $n \geq 2$ be an integer. If $G$ is a spanning subgraph of $K_{n,n}$ with
minimum degree $\delta(G)\ge \lfloor (2n+1)/3 \rfloor$, then $\tc_2(G)\le 3$.
Moreover, there exists a spanning subgraph $H$ of $K_{n,n}$ with $\delta(H) =
\lfloor (2n+1)/3 \rfloor-1$ and~$\tc_2(H) \geq 4$.
\end{theorem}

\section{Lower-bound construction}\label{sec:lower}
The next lemma shows that the minimum degree condition in Theorem~\ref{thm:main} is tight.

\begin{lemma}\label{lem:lower} 
For every integer $n \geq 2$, there is a spanning subgraph $G$ of $K_{n,n}$ with
minimum degree $\delta(G) = \lfloor (2n+1)/3 \rfloor-1$ and~$\tc_2(G)\geq 4$.
\end{lemma}

\begin{proof}
We construct a graph $G$ with minimum degree $\delta(G)=\lfloor(2n+1)/3\rfloor
-1$ and give a $2$-edge-colouring for which at least four monochromatic
components are needed to cover~$V(G)$.
Let $X$ and $Y$ be disjoint sets with $|X|=|Y|=n$ and choose partitions
\[
X=X_1\mathbin{\dot\cup}X_2\mathbin{\dot\cup}X_3
  \mathbin{\dot\cup}\{x_{\mix},x_{\red}\}
\quad\text{and}\quad
Y=Y_1\mathbin{\dot\cup}Y_2\mathbin{\dot\cup}Y_3
  \mathbin{\dot\cup}\{y_{\mix},y_{\blue}\},
\]
such that the sizes of the given sets are as below.
\[
|X_1|=|X_2|=|Y_1|=|Y_2|
 =\left\lfloor\frac{n-2}{3}\right\rfloor,
\qquad\text{and}\qquad
|X_3|=|Y_3|
 =n-2-2\left\lfloor\frac{n-2}{3}\right\rfloor.
\]
Let $G=(X,Y;E)$. The edge set $E$ consists precisely of all edges between
$X_1$ and $Y_2\cup Y_3$, between $X_2$ and $Y_1\cup Y_3$, and between $X_3$
and $Y_1\cup Y_2$, together with all edges between $x_{\mix}$ and
$Y_1\cup Y_3$, between $x_{\red}$ and $Y_2\cup Y_3$, between $y_{\mix}$ and
$X_1\cup X_3$, and between $y_{\blue}$ and $X_2\cup X_3$.

Give colour blue to all edges incident with $y_{\blue}$, all edges between
$y_{\mix}$ and $X_1$, all edges between $x_{\mix}$ and $Y_3$, and all edges
between $X_2\cup X_3$ and $Y_1\cup Y_2$; colour every remaining edge red.
Figure~\ref{fig:G_lower_bound} depicts this construction.
\begin{figure}[htbp]
    \centering
    \usetikzlibrary{calc}
\tikzset{every picture/.style={line width=0.75pt}}
\begin{tikzpicture}[x=1pt,y=1pt,yscale=-1,xscale=1]

\definecolor{myblue}{rgb}{0.29, 0.56, 0.89}
\definecolor{myred}{rgb}{0.82, 0.01, 0.11}

\coordinate (GL1) at (50, 50);
\coordinate (GL2) at (50, 110);
\coordinate (GL3) at (50, 170);
\coordinate (GR1) at (125, 50);
\coordinate (GR2) at (125, 110);
\coordinate (GR3) at (125, 170);
\coordinate (GLtop) at (0, 80);
\coordinate (GLbot) at (0, 140);
\coordinate (GRtop) at (175, 80);
\coordinate (GRbot) at (175, 140);

\newcommand{\LBedgeBand}[3]{%
  \path[draw=#1, line width=0.1pt, fill=#1, fill opacity=0.15,
        line join=round]
      ($(#2)!6pt!90:(#3)$) --
    ($(#3)!6pt!-90:(#2)$) --
    ($(#3)!6pt!90:(#2)$) --
    ($(#2)!6pt!-90:(#3)$) -- cycle;
}
\newcommand{\smallBedgeBand}[3]{%
  \path[draw=#1, line width=0.1pt, fill=#1, fill opacity=0.15,
        line join=round]
    ($(#2)!0pt!90:(#3)$) --
    ($(#3)!21pt!-70:(#2)$) --
    ($(#3)!21pt!70:(#2)$) --
    ($(#2)!0pt!-90:(#3)$) -- cycle;
}

\LBedgeBand{myred}{GL1}{GR1}
\LBedgeBand{myred}{GL1}{GR2}

\LBedgeBand{myred}{GL2}{GR1}
\LBedgeBand{myblue}{GL2}{GR3}

\LBedgeBand{myblue}{GL3}{GR2}
\LBedgeBand{myblue}{GL3}{GR3}

\smallBedgeBand{myblue}{GLtop}{GL1}
\smallBedgeBand{myred}{GLtop}{GL3}

\smallBedgeBand{myblue}{GLbot}{GL2}
\smallBedgeBand{myblue}{GLbot}{GL3}

\smallBedgeBand{myred}{GRtop}{GR1}
\smallBedgeBand{myred}{GRtop}{GR2}

\smallBedgeBand{myblue}{GRbot}{GR1}
\smallBedgeBand{myred}{GRbot}{GR3}

\node[circle, draw=black, fill=white, text=black, minimum size=42pt,
      inner sep=0pt, font=\large] at (GL1)
      {$X_1$};
\node[circle, draw=black, fill=white, text=black, minimum size=42pt,
      inner sep=0pt, font=\large] at (GL2)
      {$X_2$};
\node[circle, draw=black, fill=white, text=black, minimum size=42pt,
      inner sep=0pt, font=\large] at (GL3)
      {$X_3$};

\node[circle, draw=black, fill=white, text=black, minimum size=42pt,
      inner sep=0pt, font=\large] at (GR1)
      {$Y_3$};
\node[circle, draw=black, fill=white, text=black, minimum size=42pt,
      inner sep=0pt, font=\large] at (GR2)
      {$Y_2$};
\node[circle, draw=black, fill=white, text=black, minimum size=42pt,
      inner sep=0pt, font=\large] at (GR3)
      {$Y_1$};

\node[circle, fill=black, inner sep=1.5pt, label=left:{$y_{\mix}$}]
  at (GLtop) {};
\node[circle, fill=black, inner sep=1.5pt, label=left:{$y_{\blue}$}]
  at (GLbot) {};
\node[circle, fill=black, inner sep=1.5pt, label=right:{$x_{\red}$}]
  at (GRtop) {};
\node[circle, fill=black, inner sep=1.5pt, label=right:{$x_{\mix}$}]
  at (GRbot) {};

\end{tikzpicture}
    \caption{The lower-bound construction of a $2$-edge-coloured spanning
    subgraph $G$ of $K_{n,n}$ with $\delta(G)=\lfloor(2n+1)/3\rfloor-1$ and
    $\tc_2(G)\ge4$.}
	\label{fig:G_lower_bound}
\end{figure}
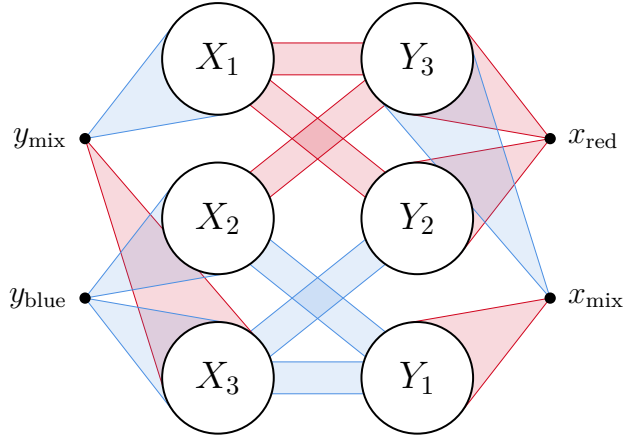

Now let us check that $\delta(G) = \lfloor (2n+1)/3 \rfloor - 1$. Put
$a:=\left\lfloor\frac{n-2}{3}\right\rfloor$ and $b:=n-2-2a$. Note that each of
the four special vertices $x_{\mix}$, $x_{\red}$, $y_{\mix}$ and $y_{\blue}$ has
degree $a+b$, every vertex in $X_1\cup X_2\cup Y_1\cup Y_2$ has degree $a+b+1$,
whereas every vertex in $X_3\cup Y_3$ has degree $2a+2$. Since $a\leq b \leq
a+2$, we have $2a+2\geq a+b$. Therefore, the minimum degree is attained by all
four special vertices; if $n\equiv1\pmod 3$, it is also attained by every
vertex in $X_3\cup Y_3$. Consequently,
\[
\delta(G)=a+b
=n-2-\left\lfloor\frac{n-2}{3}\right\rfloor
=\left\lfloor\frac{2n+1}{3}\right\rfloor-1.
\]
Finally, the red components containing $x_{\mix}$, $x_{\red}$, $y_{\mix}$,
and $y_{\blue}$, respectively, have vertex sets
\[
\{x_{\mix}\}\cup Y_1,\qquad
\{x_{\red}\}\cup X_1\cup X_2\cup Y_2\cup Y_3,\qquad
\{y_{\mix}\}\cup X_3,\qquad
\{y_{\blue}\},
\]
whereas the corresponding blue components have vertex sets
\[
\{x_{\mix}\}\cup Y_3,\qquad
\{x_{\red}\},\qquad
\{y_{\mix}\}\cup X_1,\qquad
\{y_{\blue}\}\cup X_2\cup X_3\cup Y_1\cup Y_2.
\]
In each colour these four components are distinct. Hence no monochromatic
component contains two special vertices, so every monochromatic cover of
$V(G)$ uses at least four components.
\end{proof}

\section{Covering with three monochromatic components}

In view of Lemma~\ref{lem:lower}, it remains to prove the following upper bound.

\begin{lemma}\label{lem:upper_bound}
Let $n$ be a positive integer, and let $G=(X,Y;E)$ be a bipartite graph with
$|X|=|Y|=n$. If $\delta(G)\ge \lfloor(2n+1)/3\rfloor$, then $\tc_2(G)\le 3$.
\end{lemma}
For brevity, define
\[
\delta^*(n):=\left\lfloor\frac{2n+1}{3}\right\rfloor.
\]
A direct calculation according to the residue of $n$ modulo $3$ gives the
following fact.
\begin{fact}\label{fact:simple}
For any positive integer $n$, we have 
\[
\left\lceil\frac{\delta^*(n)}{2}\right\rceil+\delta^*(n)\ge n,
\qquad
3\left\lceil\frac{\delta^*(n)}{2}\right\rceil\ge n,
\qquad
3\left(\left\lfloor\frac{\delta^*(n)}{2}\right\rfloor+1\right)>n.
\]
\end{fact}

Given a red--blue colouring of $E(G)$ and a vertex $v\in V(G)$, let $N_R(v)$
and $N_B(v)$ be the red and blue neighbourhoods of $v$, respectively, and put
$d_R(v):=|N_R(v)|$ and $d_B(v):=|N_B(v)|$. For $W\subseteq V(G)$, let
$N_R(v,W):=N_R(v)\cap W$ and $d_R(v,W):=|N_R(v,W)|$, and define the analogous
blue notation. We identify each component with its vertex set.

A vertex is \emph{red-dominant} if it is incident to at least as many red edges
as blue edges. Furthermore, we say a vertex $v$ is a \emph{red neighbour} of $u$ if
they are connected by a red edge. The terms \emph{blue-dominant} and
\emph{blue neighbour} are defined analogously. 

\subsection{Overview of the proof}

The proof is organised as a sequence of reductions. We first establish a few
configurations from which a cover with at most three monochromatic components
follows quickly. Proposition~\ref{prop:5} shows that it is enough to cover one
side of the bipartition with at most two monochromatic components and
Proposition~\ref{prop:7} gives two further stopping criteria:
item~\ref{prop:7-item1} applies when one monochromatic component contains at
least $\delta^*(n)$ vertices on one side, while item~\ref{prop:7-item2} applies
when at most three components of the same colour cover one side.

We then classify the vertices of $X$ according to their dominant colour. If
three distinct red components contain red-dominant vertices, each of them has a
large `footprint' on $Y$, and the three footprints already cover $Y$; the same
holds in blue. We may therefore focus on at most two red components $R_1$ and
$R_2$ containing all red-dominant vertices of $X$ and at most two blue
components $B_1$ and $B_2$ containing all blue-dominant vertices of $X$.
Together, these four components cover $X$ and form the skeleton of the argument.
We record which vertices belong to the selected red components but not to the
selected blue components (the ``red-only'' sets $X_{\R},Y_{\R}$), which belong
to the selected blue components but not to the selected red components (the
``blue-only'' sets $X_{\B},Y_{\B}$), and which vertices of $Y$ lie outside all
four components (the set $Y_{\out}$).

The proof splits according to whether each colour really needs two components in
this skeleton. If one colour needs fewer than two components, then all immediate
cases reduce to our stopping criteria, and the only configuration left has one
red component and two blue components (up to exchanging the colours). Two
suitably chosen vertices in the distinct blue components give an upper bound on
$|R_1\cap Y|+|Y_{\out}|$, while a red-only vertex gives a lower bound differing
from it by at most one. The exact value of $\delta^*(n)$ leaves only extremal
equality cases (including the one-unit exceptional case when $n\equiv2\pmod 3$).
In those cases the forced adjacencies either place all of $Y_{\out}$ in one
further blue component or make three red components cover $Y$. Either outcome
is one of our stopping criteria.

It remains to consider the case in which both colours need exactly two
components. Claims~\ref{claim:9}--\ref{claim:11} successively remove every less
rigid arrangement. They show that none of the four red-only and blue-only sets
can be empty, that such a set cannot be spread over both selected components of
its colour, and that the red-only sets on $X$ and $Y$ must lie in opposite red
components; the analogous statement holds in blue. Thus, after relabelling, the
only possible obstruction has a crossed form: $X_{\R}$ lies in $R_1$ while
$Y_{\R}$ lies in $R_2$, and $X_{\B}$ lies in $B_1$ while $Y_{\B}$ lies in $B_2$.
Common-neighbour paths also show that any vertex of $Y_{\out}$ would immediately
yield a cover by three components, so the alleged obstruction has
$Y_{\out}=\emptyset$.

This crossed configuration is too rigid to satisfy the minimum-degree condition.
Unless the proof has already finished, we can choose two vertices of $X_{\R}$
with disjoint blue neighbourhoods in $Y_{\R}$ and, symmetrically, two vertices
of $X_{\B}$ with disjoint red neighbourhoods in $Y_{\B}$. The degree
requirements for these four vertices, together with the large footprints
supplied by dominant vertices in $R_2$ and $B_2$, demand more room in $Y$ than
is available. More precisely, their combined bounds force one of $Y_{\R}$ and
$Y_{\B}$ to be empty, contradicting the structure forced by the three claims.
Hence the final obstruction cannot occur, and one of the stopping criteria must
always apply. This proves that three monochromatic components cover $V(G)$.

Before proving Lemma~\ref{lem:upper_bound}, we state and prove some auxiliary
results that will provide the structure for our proof.

\subsection{Structuring the covering}

For a fixed $2$-edge-colouring $\varphi$ of $E(G)$, let $\tc_2(G,\varphi)$
denote the least number of monochromatic components in the colouring $\varphi$
whose vertex sets cover $V(G)$. Thus, $\tc_2(G)=\max\{\tc_2(G,\varphi) :
\varphi\text{ is a $2$-edge-colouring of $G$}\}$. We begin with a simple observation.

\begin{fact}\label{fact:auxiliary}
    Let $G$ be a $2$-edge-coloured graph, and let $S\subseteq V(G)$ be such that any two vertices in $S$ are connected by a monochromatic path. Then $S$ is covered by a single monochromatic component. 
\end{fact}

\begin{proof}
Consider the following auxiliary $2$-edge-coloured complete graph on $S$: the colour of an edge between two vertices coincides with the colour of a monochromatic path between them in $G$; we choose arbitrarily if both colours are available. The statement follows from the fact that every $2$-edge-coloured complete graph has a monochromatic spanning tree. 
\end{proof}

We continue with a useful
criterion for bounding this fixed-colouring covering number.

\begin{proposition}\label{prop:5}
Let $n$ be a positive integer, and let $G=(X,Y;E)$ be a bipartite graph with
$|X|=|Y|=n$ such that $\delta(G)\ge \lfloor n/2\rfloor+1$. Let $\varphi$ be a
$2$-edge-colouring of $E(G)$. If at most two monochromatic components cover $X$ or
cover $Y$, then $\tc_2(G,\varphi)\le 3$.

\end{proposition}
\begin{proof}

First, let us consider the case where there are at most two monochromatic components of
the same colour covering one side of the bipartition of $G$. Without
loss of generality, let $R_1$ and $R_2$ be two red components covering $X$, one of which may be empty.
If $Y \setminus (R_1 \cup R_2) = \emptyset$, then
$\tc_2(G,\varphi) \le 2$.
Otherwise, every vertex in $Y \setminus (R_1 \cup R_2)$ has only blue neighbours
in $X$, as they are in neither $R_1$ nor $R_2$. Since $\delta(G) \ge \lfloor
\frac{n}{2}\rfloor + 1$, any two vertices in $Y \setminus (R_1 \cup R_2)$ have a
common blue neighbour in $X$. Therefore, there is a blue component $B_1$
covering $Y \setminus (R_1 \cup R_2)$. This, together with $R_1$ and $R_2$,
yields three monochromatic components covering $V(G)$, hence
$\tc_2(G,\varphi) \le 3$.
See Figure~\ref{fig:prop5_same}.
\begin{figure}[htbp]
    \centering
    \tikzset{every picture/.style={line width=0.75pt}}
\begin{tikzpicture}[x=1pt,y=1pt,yscale=-1,xscale=1]

\definecolor{myred}{rgb}{0.82, 0.01, 0.11}
\definecolor{myblue}{rgb}{0.29, 0.56, 0.89}
\definecolor{mypink}{rgb}{0.96, 0.76, 0.85}
\definecolor{mycyan}{rgb}{0.80, 0.95, 0.98}

\node at (50, 10) {$X$};
\node at (100, 10) {$Y$};
\draw[thick] (75, 20) -- (75, 135);

\draw[myred, thick, rounded corners=10pt, fill=mypink, fill opacity=0.3] (30, 30) rectangle (120, 60);
\draw[myred, thick, fill=mypink, fill opacity=0.3, rounded corners=10pt] 
  (120, 100) -- (70, 100) -- (70, 135) -- (30, 135) -- (30, 70) -- (70, 70)  -- (120, 70) -- cycle;

\node[myred] at (20, 45) {$R_1$};
\node[myred] at (20, 85) {$R_2$};

  \node[circle, fill=myblue, inner sep=1.5pt] at (100, 109.5) {};
  \node[myblue] at (100, 116) {$\vdots$};
  \node[circle, fill=myblue, inner sep=1.5pt] at (100, 128.5) {};

\begin{scope}[xshift=200pt]
\node at (50, 10) {$X$};
\node at (100, 10) {$Y$};
\draw[thick] (75, 20) -- (75, 120);

\draw[myred, thick, rounded corners=10pt, fill=mypink, fill opacity=0.3] (30, 30) rectangle (120, 60);
\draw[myred, thick, fill=mypink, fill opacity=0.3, rounded corners=10pt] 
  (120, 100) -- (70, 100) -- (70, 135) -- (30, 135) -- (30, 70) -- (70, 70)  -- (120, 70) -- cycle;

\node[myred] at (20, 45) {$R_1$};
\node[myred] at (20, 85) {$R_2$};

\draw[myblue, thick, fill=mycyan, fill opacity=0.3, rounded corners=10pt] 
  (120, 135) -- (35, 135) -- (35, 40) -- (70, 40) -- (70, 105) -- (120, 105) -- cycle;

\begin{scope}
  \clip[rounded corners=10pt] (120, 135) -- (35, 135) -- (35, 40) -- (70, 40) -- (70, 105) -- (120, 105) -- cycle;
  
  \begin{scope}
    \clip[rounded corners=10pt] (30, 30) rectangle (120, 60);
    \begin{pgfinterruptboundingbox}
      \fill[purple, opacity=0.4] (20, 20) rectangle (130, 140);
      \foreach \i in {-100, -95, ..., 140} {
        \draw[white, thick] (\i, 20) -- (\i+120, 140);
      }
    \end{pgfinterruptboundingbox}
  \end{scope}

  \begin{scope}
    \clip[rounded corners=8pt] (30, 70) rectangle (70, 135);
    \begin{pgfinterruptboundingbox}
      \fill[purple, opacity=0.4] (20, 20) rectangle (130, 140);
      \foreach \i in {-100, -95, ..., 140} {
        \draw[white, thick] (\i, 20) -- (\i+120, 140);
      }
    \end{pgfinterruptboundingbox}
  \end{scope}
  
  \node[circle, fill=myblue, inner sep=1.5pt] at (100, 109.5) {};
  \node[myblue] at (100, 116) {$\vdots$};
  \node[circle, fill=myblue, inner sep=1.5pt] at (100, 128.5) {};
\end{scope}

\node[myblue, anchor=west] at (120, 120) {$B_1$};
\end{scope}

\node[black] at (150, 120) {\small{$Y \setminus (R_1 \cup R_2)$}};

\draw[thick, ->] (165, 65) -- (195, 65);

\end{tikzpicture}
    \caption{Schematic representation for red components $R_1$ and $R_2$ covering~$X$.}
    \label{fig:prop5_same}
\end{figure}
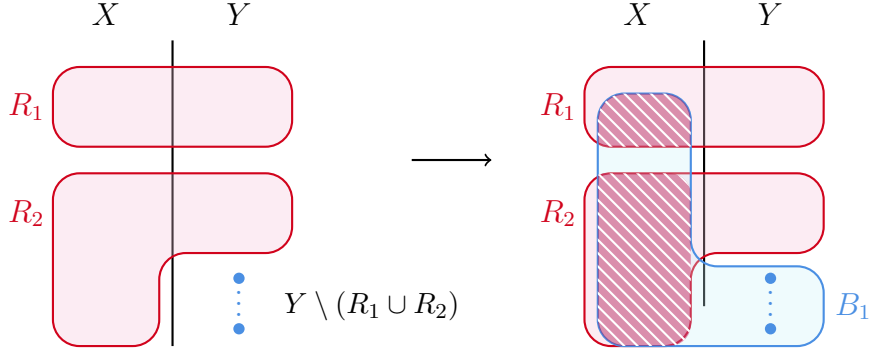

Now suppose that a red component $R_1$ and a blue component $B_1$ together
cover $X$. Every vertex in $Y\setminus(R_1\cup B_1)$ has only blue neighbours
in $R_1$ and red neighbours in $B_1$. Because any two such vertices have a
common neighbour in $X$, they are joined by a red path of length two through
$B_1$ or a blue path of length two through $R_1$. By Fact~\ref{fact:auxiliary},
all vertices of $Y\setminus(R_1\cup B_1)$ belong to one monochromatic component
$C$, which together with $R_1$ and $B_1$ implies
$\tc_2(G,\varphi) \le 3$. See Figure~\ref{fig:prop5_distinct}.
\begin{center}
\begin{minipage}{\textwidth}
    \centering
    \captionsetup{hypcap=false}
    \tikzset{every picture/.style={line width=0.75pt}}
\begin{tikzpicture}[x=1pt,y=1pt,yscale=-1,xscale=1]

\definecolor{myred}{rgb}{0.82, 0.01, 0.11}
\definecolor{myblue}{rgb}{0.29, 0.56, 0.89}
\definecolor{mygreen}{rgb}{0.13, 0.55, 0.13}
\definecolor{mypink}{rgb}{0.96, 0.76, 0.85}
\definecolor{mycyan}{rgb}{0.80, 0.95, 0.98}

\node at (50, 10) {$X$};
\node at (100, 10) {$Y$};
\draw[thick] (75, 20) -- (75, 120);

\draw[myred, thick, rounded corners=10pt, fill=mypink, fill opacity=0.3] (30,35) -- (120,35) -- (120,70) -- (100,70) -- (90,50) -- (60,50) -- (50,70) -- (30,70) -- cycle;
\draw[myblue, thick, rounded corners=10pt, fill=mycyan, fill opacity=0.3] (30,95) -- (120,95) -- (120,60) -- (100,60) -- (90,80) -- (60,80) -- (50,60) -- (30,60) -- cycle;

\begin{scope}
  \clip[rounded corners=10pt] (30,35) -- (120,35) -- (120,70) -- (100,70) -- (90,50) -- (60,50) -- (50,70) -- (30,70) -- cycle;
  \clip[rounded corners=10pt] (30,95) -- (120,95) -- (120,60) -- (100,60) -- (90,80) -- (60,80) -- (50,60) -- (30,60) -- cycle;
  \begin{pgfinterruptboundingbox}
    \fill[purple, opacity=0.4] (20, 50) rectangle (130, 80);
    \foreach \i in {10, 15, ..., 140} {
      \draw[white, thick] (\i, 50) -- (\i+30, 80);
    }
  \end{pgfinterruptboundingbox}
\end{scope}

\node[myred] at (20, 45) {$R_1$};
\node[myblue] at (20, 85) {$B_1$};
\node[black] at (150, 120) {\small{$Y \setminus (R_1 \cup B_1)$}};

\node[circle, fill=black, inner sep=1.5pt] at (100, 110) {};
\node[black] at (100, 116) {$\vdots$};
\node[circle, fill=black, inner sep=1.5pt] at (100, 129) {};

\begin{scope}[xshift=200pt]
\node at (50, 10) {$X$};
\node at (100, 10) {$Y$};
\draw[thick] (75, 20) -- (75, 120);

\draw[myred, thick, fill=mypink, fill opacity=0.3, rounded corners=15pt] 
  (120, 135) -- (33, 135) -- (33, 75) -- (62, 75) -- (62, 105) -- (120, 105) -- cycle;

\begin{scope}
  \clip[rounded corners=15pt] (120, 135) -- (33, 135) -- (33, 75) -- (62, 75) -- (62, 105) -- (120, 105) -- cycle;
  \clip[rounded corners=10pt] (30,95) -- (120,95) -- (120,60) -- (100,60) -- (90,80) -- (60,80) -- (50,60) -- (30,60) -- cycle;
  \clip (20, 20) rectangle (75, 140);
  \begin{pgfinterruptboundingbox}
    \fill[purple, opacity=0.4] (20, 20) rectangle (130, 140);
    \foreach \i in {-100, -95, ..., 140} {
      \draw[white, thick] (\i, 20) -- (\i+120, 140);
    }
  \end{pgfinterruptboundingbox}
\end{scope}

\draw[myred, thick, rounded corners=10pt, fill=mypink, fill opacity=0.3] (30,35) -- (120,35) -- (120,70) -- (100,70) -- (90,50) -- (60,50) -- (50,70) -- (30,70) -- cycle;
\draw[myblue, thick, rounded corners=10pt, fill=mycyan, fill opacity=0.3] (30,95) -- (120,95) -- (120,60) -- (100,60) -- (90,80) -- (60,80) -- (50,60) -- (30,60) -- cycle;

\begin{scope}
  \clip[rounded corners=10pt] (30,35) -- (120,35) -- (120,70) -- (100,70) -- (90,50) -- (60,50) -- (50,70) -- (30,70) -- cycle;
  \clip[rounded corners=10pt] (30,95) -- (120,95) -- (120,60) -- (100,60) -- (90,80) -- (60,80) -- (50,60) -- (30,60) -- cycle;
  \begin{pgfinterruptboundingbox}
    \fill[purple, opacity=0.4] (20, 50) rectangle (130, 80);
    \foreach \i in {10, 15, ..., 140} {
      \draw[white, thick] (\i, 50) -- (\i+30, 80);
    }
  \end{pgfinterruptboundingbox}
\end{scope}

\node[myred] at (20, 45) {$R_1$};
\node[myblue] at (20, 85) {$B_1$};
\node[myred] at (127, 120) {$C$};

\node[circle, fill=black, inner sep=1.5pt] at (100, 110) {};
\node[black] at (100, 116) {$\vdots$};
\node[circle, fill=black, inner sep=1.5pt] at (100, 129) {};

\end{scope}

\draw[thick, ->] (165, 65) -- (195, 65);

\end{tikzpicture}

    
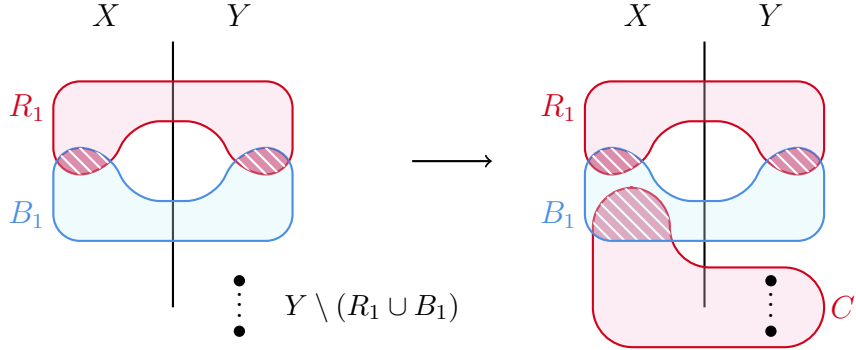
\captionof{figure}{Schematic representation for components $R_1$ and $B_1$
    with distinct colours covering~$X$.}
    \label{fig:prop5_distinct}
\end{minipage}
\end{center}
\end{proof}

The preceding proposition handles configurations in which one side of the
bipartition is already covered by at most two monochromatic components. We now
establish two complementary stopping criteria tailored to the threshold
$\delta^*(n)$. They will allow us to conclude whenever either one component is
sufficiently large on one side or a small collection of components of the same
colour covers one side.
\begin{proposition}\label{prop:7}
Let $n$ be a positive integer, and let $G=(X,Y;E)$ be a bipartite graph with
$|X|=|Y|=n$ such that $\delta(G)\ge\delta^*(n)$. Let $\varphi$ be a
$2$-edge-colouring of $E(G)$. Then the following statements hold.
\begin{enumerate}[(i)]
\item If there is a monochromatic component that covers at least $\delta^*(n)$
vertices on one side of the bipartition of $G$, then $\tc_2(G,\varphi) \le 3$;
\label{prop:7-item1}
\item If there are at most three components of the same colour that cover one
side of the bipartition of $G$, then $\tc_2(G,\varphi) \le 3$.\label{prop:7-item2}
\end{enumerate}
\end{proposition}

\begin{proof}
We start by proving item~\ref{prop:7-item1}. By symmetry, let $R_1$ be a red component
covering at least $\delta^*(n)$ vertices of $Y$. For $n\le 2$, the result is
immediate, so assume that $n\ge 3$. Since $\delta^*(n)>n/2$,
Proposition~\ref{prop:5} applies whenever one side is covered by at most two
monochromatic components.

First suppose that there exists a red-dominant vertex $u \in X \setminus
R_1$ and let $R_2$ be the red component containing $u$. Then, $|Y\cap R_2| \ge
\lceil\delta^*(n)/2\rceil$. Therefore, from Fact~\ref{fact:simple} we have
$|Y\cap(R_1\cup R_2)|\ge\delta^*(n)+\lceil\delta^*(n)/2\rceil\ge n$.
Thus two red components cover~$Y$, and Proposition~\ref{prop:5} implies
$\tc_2(G,\varphi) \le 3$.

Now suppose there is no red-dominant vertex in $A:=X\setminus R_1$. Then
every vertex $x\in A$ satisfies $d_B(x)>d_R(x)$ and since $d(x)\ge \delta^*(n)$,
we have $d_B(x)\ge \left\lfloor \delta^*(n)/2\right\rfloor+1$. Thus, every
blue component that intersects $A$ contains at least
$\left\lfloor\delta^*(n)/2\right\rfloor+1$ vertices of $Y$. Consequently, at
most two blue components intersect $A$: if three did, their pairwise disjoint
intersections with $Y$ would together contain more than $n$ vertices, which is a contradiction by Fact~\ref{fact:simple}.

If at most one blue component meets $A$, then $X$ is covered by $R_1$ and at
most one blue component and Proposition~\ref{prop:5} implies
$\tc_2(G,\varphi)\le 3$. Thus, we may assume that there are exactly two blue
components $B_u$ and $B_v$ covering $A$.

If every pair of vertices $x\in B_u\cap A$ and $z\in B_v\cap A$ has a common red
neighbour, then all vertices of $A$ lie in a single red component. Hence this
red component together with $R_1$ covers $X$, and Proposition~\ref{prop:5}
implies $\tc_2(G,\varphi)\le 3$. Thus, we may choose vertices $u\in B_u\cap A$ and $v\in
B_v\cap A$ with no common red neighbour. Since $B_u$ and $B_v$ are distinct blue
components, $u$ and $v$ have no common blue neighbour either. Therefore every
common neighbour of $u$ and $v$ is joined to them by edges of distinct colours.

Moreover, $u$ and $v$ have no common neighbour in $Y\cap R_1$: indeed, every
edge from $u$ or $v$ to a vertex of $Y\cap R_1$ is blue, because
$u,v\notin R_1$, and $u$ and $v$ have no common blue neighbour. Hence all
common neighbours of $u$ and $v$ lie in $Y\setminus R_1$.

Let $S$ denote the set of common neighbours of $u$ and $v$ in $Y \setminus R_1$
and $T$ the set of vertices of $Y \setminus R_1$ adjacent to at most one of $u$
and $v$. The next step is to show that $|T| \le 1$. For that, first observe
that we have
\begin{align*}
|S| &\ge d(u) + d(v) - n \ge 2\delta^*(n) - n,\\
|S| &= n - |Y \cap R_1| - |T| \le n - \delta^*(n) - |T|.
\end{align*}
This gives the desired bound on $|T|$, which is 
\begin{equation*}
|T| \le 2n - 3\delta^*(n) \le 1.
\end{equation*}
Denote by $S_u$ the blue neighbours of $u$ in $S$ and $S_v$ the blue neighbours
of $v$ in $S$, so that $S=S_u\mathbin{\dot\cup}S_v$.

If $T=\emptyset$, then the components $R_1$, $B_u$ and $B_v$ cover $Y$. They
also cover $X$, because $B_u$ and $B_v$ cover $X\setminus R_1$. Hence
$\tc_2(G,\varphi)\le 3$. Thus, we may assume that $T=\{t\}$.

We first show that $t$ is adjacent to at least one of $u$ and $v$. Suppose not.
Since $u$ and $v$ have no common blue neighbour, they have disjoint blue
neighbourhoods in $R_1\cap Y$. Hence
\[
    2\delta^*(n)\le d(u)+d(v)\le 2|S|+|R_1\cap Y|
        =2(n-|R_1\cap Y|-1)+|R_1\cap Y|
        \le 2n-\delta^*(n)-2,
\]
which contradicts $3\delta^*(n)>2n-2$. Therefore, without loss of generality,
suppose the sole vertex $t\in T$ is adjacent to $u$ and not to $v$.

If the edge $ut$ is blue, then $B_u$ covers $T$ and $S_u$. Since $B_v$ covers
$S_v$, the components $R_1$, $B_u$ and $B_v$ cover $Y$. They also cover $X$,
because $B_u$ and $B_v$ cover $X\setminus R_1$. Hence $\tc_2(G,\varphi)\le 3$.

It remains to consider the case where $ut$ is red. Let $R_u$ and $R_v$ be the
red components containing $u$ and $v$, respectively. Then $R_u$ covers
$T\cup S_v$, while $R_v$ covers $S_u$. Therefore the three red components
$R_1$, $R_u$ and $R_v$ cover $Y$.

If these three red components also cover $X$, then we are done. Thus, suppose
that there exists
\[
    x\in X\setminus (R_1\cup R_u\cup R_v).
\]
Since $R_1$, $R_u$ and $R_v$ cover $Y$, the vertex $x$ sends no red edge to
$Y$. Hence all edges incident to $x$ are blue. Suppose, without loss of
generality, that $x\in B_u$. Since $B_u$ and $B_v$ are distinct blue components,
$x$ has no blue neighbour in $B_v\cap Y$, and therefore no neighbour at all in
$B_v\cap Y$. Consequently,
\[
    \delta^*(n) \le d(x) \le n-|B_v\cap Y|.
\]
On the other hand, $v$ is blue-dominant, and so
\[
    |B_v\cap Y| \ge d_B(v) \ge \left\lceil \frac{\delta^*(n)}{2}\right\rceil.
\]
Thus,
\[
    \delta^*(n) \le d(x) \le n-|B_v\cap Y|
        \le n-\left\lceil \frac{\delta^*(n)}{2}\right\rceil
        \le \delta^*(n),
\]
where the last inequality follows from Fact~\ref{fact:simple}. Hence equality
holds throughout. In particular, $x$ is adjacent to every vertex of $Y\setminus
B_v$, and all these edges are blue. Therefore $Y\setminus B_v\subseteq B_u$, and
so $B_u\cup B_v$ covers $Y$. Since $B_u$ and $B_v$ cover $X\setminus R_1$, the
three components $R_1$, $B_u$ and $B_v$ cover $V(G)$. Hence $\tc_2(G,\varphi)\le
3$, which proves item~\ref{prop:7-item1}.

Item~\ref{prop:7-item2} follows from item~\ref{prop:7-item1}. Indeed, suppose
without loss of generality that at most three red components cover
$X$. If they also cover $Y$, then $\tc_2(G,\varphi)\le 3$. Otherwise, there is a vertex
$y\in Y$ outside all of them. Every edge from $y$ to $X$ is then blue, so the
blue component containing $y$ covers $N(y)$ and hence at least
$d(y)\ge \delta^*(n)$ vertices of $X$. Item~\ref{prop:7-item1} now gives
$\tc_2(G,\varphi)\le 3$, which finishes the proof.
\end{proof}

\subsection{Proof of Lemma~\ref{lem:upper_bound}}

With Propositions~\ref{prop:5} and~\ref{prop:7} in hand, we now prove the upper
bound for the minimum degree that guarantees that three monochromatic components
are always enough. The argument selects small collections of red and blue
components that contain all dominant vertices on one side of the bipartition.
The preceding propositions dispose of the flexible configurations, leaving only
a rigid arrangement that contradicts the minimum-degree condition.

\begin{proof}
Fix an arbitrary $2$-edge-colouring $\varphi$ of $E(G)$.
Since the result is trivial for $n\le 2$, we may assume that $n\ge 3$. In
particular, $\delta^*(n)>n/2$, and so any two vertices in the same side of the
bipartition have a common neighbour. Moreover, by Proposition~\ref{prop:5},
whenever one side of $G$ is covered by at most two monochromatic components, we
have $\tc_2(G,\varphi)\le 3$. By Proposition~\ref{prop:7}~\ref{prop:7-item1}, we may
also assume throughout the proof that no vertex has at least $\delta^*(n)$
neighbours of the same colour. In particular, since $d(v)\ge \delta^*(n)$ for
every vertex $v$, every vertex is incident to at least one red edge and at least
one blue edge.
 
Suppose first that there are three distinct red components, each containing a
red-dominant vertex of $X$. Since each of these components contains a
red-dominant vertex of $X$, each of them contains at least $\lceil
\delta^*(n)/2\rceil$ vertices of $Y$. Hence, by Fact~\ref{fact:simple}, these
three red components cover $Y$, and Proposition~\ref{prop:7}~\ref{prop:7-item2}
gives $\tc_2(G,\varphi)\le 3$. Thus, we may assume that at most two red components cover
all red-dominant vertices of $X$. The same argument with the colours
interchanged shows that we may also assume that at most two blue components
cover all blue-dominant vertices of $X$.

Let $\mathcal R$ be a minimal collection of at most two red components covering
all red-dominant vertices of $X$, and let $\mathcal B$ be a minimal collection
of at most two blue components covering all blue-dominant vertices of $X$. We
split our proof into two parts, depending on the sizes of $\mathcal R$ and
$\mathcal B$.

We write the elements of $\mathcal R$ as $R_1,R_2$ and those of $\mathcal B$ as
$B_1,B_2$, allowing either or both sets to be empty when the
corresponding collection has fewer than two elements.

The following subset of $Y$ plays an important role in our proof:
\begin{equation*}
Y_{\out} = Y \setminus (R_1 \cup R_2 \cup B_1 \cup B_2).
\end{equation*}
We also consider the vertices that belong to the selected components of exactly
one colour:
\begin{align*}
X_{\R} &= X\cap\bigl((R_1\cup R_2)\setminus(B_1\cup B_2)\bigr), &
Y_{\R} &= Y\cap\bigl((R_1\cup R_2)\setminus(B_1\cup B_2)\bigr),\\
X_{\B} &= X\cap\bigl((B_1\cup B_2)\setminus(R_1\cup R_2)\bigr), &
Y_{\B} &= Y\cap\bigl((B_1\cup B_2)\setminus(R_1\cup R_2)\bigr).
\end{align*}

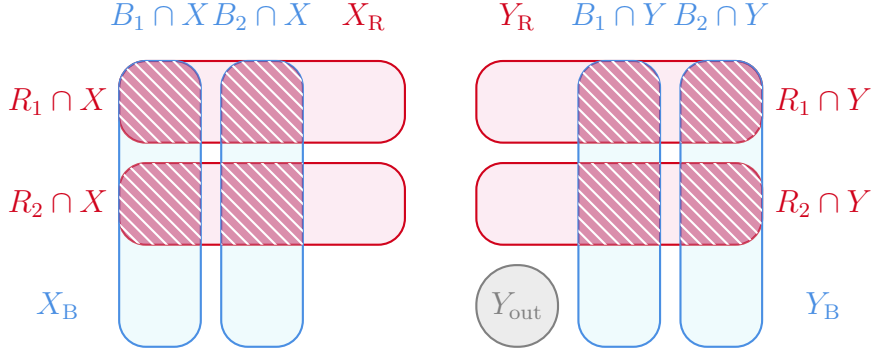
\begin{figure}[htbp]
    \centering
	\tikzset{every picture/.style={line width=0.75pt}}
\begin{tikzpicture}[scale=0.64,x=1pt,y=1pt]

\definecolor{myred}{rgb}{0.82, 0.01, 0.11}
\definecolor{myblue}{rgb}{0.29, 0.56, 0.89}
\definecolor{mypink}{rgb}{0.96, 0.76, 0.85}
\definecolor{mycyan}{rgb}{0.80, 0.95, 0.98}

\newcommand{\XRBlueRedIntersection}[2]{%
  \begin{scope}
    \clip[rounded corners=10pt] #1;
    \clip[rounded corners=10pt] #2;
    \begin{pgfinterruptboundingbox}
      \fill[purple, opacity=0.4] (-20,-20) rectangle (200,200);
      \foreach \i in {-220,-212,...,180} {
        \draw[white, thick] (\i+220,-20) -- (\i,200);
      }
    \end{pgfinterruptboundingbox}
  \end{scope}%
}

\def\rmargin{6}
\draw[myred, thick, rounded corners=10pt, fill=mypink, fill opacity=0.3]
  (0+\rmargin,120+\rmargin) rectangle (180-\rmargin,180-\rmargin);

\draw[myred, thick, rounded corners=10pt, fill=mypink, fill opacity=0.3]
  (0+\rmargin,60+\rmargin) rectangle (180-\rmargin,120-\rmargin);

\def\bmargin{6}

\draw[myblue, thick, rounded corners=10pt, fill=mycyan, fill opacity=0.3]
  (0+\bmargin, 0+\bmargin) rectangle (60-\bmargin, 180-\bmargin);

\draw[myblue, thick, rounded corners=10pt, fill=mycyan, fill opacity=0.3]
  (60+\bmargin, 0+\bmargin) rectangle (120-\bmargin, 180-\bmargin);

\XRBlueRedIntersection
  {(0+\rmargin,120+\rmargin) rectangle (180-\rmargin,180-\rmargin)}
  {(0+\bmargin,0+\bmargin) rectangle (60-\bmargin,180-\bmargin)}
\XRBlueRedIntersection
  {(0+\rmargin,120+\rmargin) rectangle (180-\rmargin,180-\rmargin)}
  {(60+\bmargin,0+\bmargin) rectangle (120-\bmargin,180-\bmargin)}
\XRBlueRedIntersection
  {(0+\rmargin,60+\rmargin) rectangle (180-\rmargin,120-\rmargin)}
  {(0+\bmargin,0+\bmargin) rectangle (60-\bmargin,180-\bmargin)}
\XRBlueRedIntersection
  {(0+\rmargin,60+\rmargin) rectangle (180-\rmargin,120-\rmargin)}
  {(60+\bmargin,0+\bmargin) rectangle (120-\bmargin,180-\bmargin)}

\node[myred] at (-30, 150) {$R_1 \cap X$};
\node[myred] at (-30, 90)  {$R_2 \cap X$};
\node[myblue] at (30, 200) {$B_1 \cap X$};
\node[myblue] at (90, 200)  {$B_2 \cap X$};

\node[myblue] at (-30, 30) {$X_{\B}$};
\node[myred] at (150, 200) {$X_{\R}$};

\begin{scope}[shift={(210, 0)}]

	\def\rmargin{6}
	\draw[myred, thick, rounded corners=10pt, fill=mypink, fill opacity=0.3]
	  (0+\rmargin,120+\rmargin) rectangle (180-\rmargin,180-\rmargin);

	\draw[myred, thick, rounded corners=10pt, fill=mypink, fill opacity=0.3]
	  (0+\rmargin,60+\rmargin) rectangle (180-\rmargin,120-\rmargin);

	\def\bmargin{6}

	\draw[myblue, thick, rounded corners=10pt, fill=mycyan, fill opacity=0.3]
	  (60+\bmargin, 0+\bmargin) rectangle (120-\bmargin, 180-\bmargin);

	\draw[myblue, thick, rounded corners=10pt, fill=mycyan, fill opacity=0.3]
	  (120+\bmargin, 0+\bmargin) rectangle (180-\bmargin, 180-\bmargin);

	\XRBlueRedIntersection
	  {(0+\rmargin,120+\rmargin) rectangle (180-\rmargin,180-\rmargin)}
	  {(60+\bmargin,0+\bmargin) rectangle (120-\bmargin,180-\bmargin)}
	\XRBlueRedIntersection
	  {(0+\rmargin,120+\rmargin) rectangle (180-\rmargin,180-\rmargin)}
	  {(120+\bmargin,0+\bmargin) rectangle (180-\bmargin,180-\bmargin)}
	\XRBlueRedIntersection
	  {(0+\rmargin,60+\rmargin) rectangle (180-\rmargin,120-\rmargin)}
	  {(60+\bmargin,0+\bmargin) rectangle (120-\bmargin,180-\bmargin)}
	\XRBlueRedIntersection
	  {(0+\rmargin,60+\rmargin) rectangle (180-\rmargin,120-\rmargin)}
	  {(120+\bmargin,0+\bmargin) rectangle (180-\bmargin,180-\bmargin)}

	\node[myred] at (210, 150) {$R_1 \cap Y$};
	\node[myred] at (210, 90)  {$R_2 \cap Y$};
	\node[myblue] at (90, 200) {$B_1 \cap Y$};
	\node[myblue] at (150, 200) {$B_2 \cap Y$};

	\node[myred] at (30, 200) {$Y_{\R}$};
	\node[myblue] at (210, 30) {$Y_{\B}$};

	\draw[gray, thick, rounded corners = 15pt, fill = gray, fill opacity=0.15]
	  (0+\bmargin, 0+\bmargin) rectangle (60-\bmargin, 60-\bmargin);
	\node[gray] at (30, 30) {$Y_{\rm out}$};
\end{scope}

\end{tikzpicture}
    \caption{Schematic representation for the sets
		$X_{\R}, X_{\B}, Y_{\R}, Y_{\B}, Y_{\out}$.}
    \label{fig:structure_XR_XB}
\end{figure}

We first dispose of the case in which one of these minimal covers has size at
most one.

\medskip
\noindent\textbf{Part 1:} $|\mathcal R|\le1$ or $|\mathcal B|\le1$.
\par\smallskip
If $\mathcal R=\emptyset$, then every vertex of $X$ is blue-dominant, and hence
$\mathcal B$ covers $X$. Since $|\mathcal B|\le 2$, Proposition~\ref{prop:5}
gives $\tc_2(G,\varphi)\le 3$. The case $\mathcal B=\emptyset$ is analogous. If both
$\mathcal R$ and $\mathcal B$ have size one, then one red component together
with one blue component covers $X$, and again Proposition~\ref{prop:5} gives
$\tc_2(G,\varphi)\le 3$. Thus, by symmetry between the colours, it remains to consider
the case in which $\mathcal R=\{R_1\}$ and $\mathcal B=\{B_1,B_2\}$. In this case,
$R_2=\emptyset$. Since every vertex of $X$ is red-dominant or blue-dominant, we
have $X\subseteq R_1\cup B_1\cup B_2$. If $X$ is covered by at most two of $R_1$, $B_1$, and $B_2$, then
Proposition~\ref{prop:5} gives $\tc_2(G,\varphi)\le 3$. Hence we may assume that all
three components are necessary to cover $X$. In particular,
\[
    X_{\R} \cap R_1\neq\emptyset,\footnote{As $R_2=\emptyset$, we have $X_{\R}\cap R_1=X_{\R}$. For the sake of consistent exposition, we prefer to write the former.}\qquad
    X_{\B} \cap B_1\neq\emptyset,\qquad
    X_{\B} \cap B_2\neq\emptyset.
\]

We first show that we may assume $Y_{\R}=\emptyset$. Indeed, suppose that
$y\in Y_{\R}$. We claim that $N(y)\subseteq R_1\cap X$. If $yx$ is red, then
$x\in R_1$. If $yx$ is blue and $x\notin R_1$, then $x$ is not red-dominant,
and hence $x$ is blue-dominant. Therefore $x\in B_1\cup B_2$, which would imply
$y\in B_1\cup B_2$, contradicting $y\in Y_{\R}$. Thus
$N(y)\subseteq R_1\cap X$, and so
\[
    |R_1\cap X|\ge d(y)\ge \delta^*(n).
\]
By Proposition~\ref{prop:7}~\ref{prop:7-item1}, this implies $\tc_2(G,\varphi)\le 3$.
Therefore, in the remaining case, we may assume that
\[
    Y_{\R}=\emptyset.
\]

If $Y_{\out}=\emptyset$, then $B_1$ and $B_2$ cover $Y$: indeed, every vertex of
$R_1\cap Y$ lies in $B_1\cup B_2$, since $Y_{\R}=\emptyset$, and every vertex
outside $R_1$ also lies in $B_1\cup B_2$, since $Y_{\out}=\emptyset$. Thus
Proposition~\ref{prop:5} gives $\tc_2(G,\varphi)\le 3$. Hence we may assume that
\[
    Y_{\out}\neq\emptyset.
\]
We next claim that we may choose vertices $u\in X_{\B}\cap B_1$ and
$v\in X_{\B}\cap B_2$ with no common red neighbour. Otherwise, every pair
$x\in X_{\B}\cap B_1$ and $z\in X_{\B}\cap B_2$ has a common red neighbour.
Since both sets are non-empty, all vertices of $X_{\B}$ lie in a single red
component. As $X\subseteq R_1\cup X_{\B}$, this component together with $R_1$
covers $X$, and Proposition~\ref{prop:5} gives
$\tc_2(G,\varphi)\le 3$.

Let $u\in X_{\B}\cap B_1$ and $v\in X_{\B}\cap B_2$ have no common red
neighbour. Since $B_1$ and $B_2$ are distinct blue components, $u$ and $v$ have
no common blue neighbour either. The red neighbours of $u$ and $v$ all lie in
$Y\setminus R_1$, and, since $u$ and $v$ have no common red neighbour, we have
\[
    d_R(u)+d_R(v)\le n-|R_1\cap Y|.
\]
Moreover, the blue neighbours of $u$ lie in $B_1\cap Y$, while the blue
neighbours of $v$ lie in $B_2\cap Y$. Since $B_1$ and $B_2$ are distinct blue
components and no vertex of $Y_{\out}$ lies in $B_1\cup B_2$, we have
\[
    d_B(u)+d_B(v)\le n-|Y_{\out}|.
\]
Hence, $2\delta^*(n)\le d(u)+d(v)\le 2n-|R_1\cap Y|-|Y_{\out}|$,
which gives
\begin{equation}\label{eq:one-red-two-blue-upper}
    |R_1\cap Y|+|Y_{\out}|\le 2(n-\delta^*(n)).
\end{equation}
On the other hand, let $x\in X_{\R} \cap R_1$. We claim that
\[
    N(x)\subseteq (R_1\cap Y)\cup Y_{\out}.
\]
Indeed, if $xy$ is red, then $y\in R_1\cap Y$. If $xy$ is blue, then
$y\notin B_1\cup B_2$, since otherwise $x$ would belong to $B_1\cup B_2$,
contrary to $x\in X_{\R}$. Since in Part 1 we have $R_2=\emptyset$, it follows
that either $y\in R_1\cap Y$ or $y\in Y_{\out}$.

Consequently,
\[
    \delta^*(n)\le d(x)\le |R_1\cap Y|+|Y_{\out}|.
\]
Together with~\eqref{eq:one-red-two-blue-upper}, this gives
\begin{equation}\label{eq:one-red-two-blue-sandwich}
    \delta^*(n)\le |R_1\cap Y|+|Y_{\out}|\le 2(n-\delta^*(n)).
\end{equation}

We have to split the remainder of the proof of Part 1 depending on the value of
$n \pmod 3$. This is needed because 
\[
2(n-\delta^*(n))=
\begin{cases}
\delta^*(n), & \text{if } n\equiv 0 \pmod 3,\\
\delta^*(n)-1, & \text{if } n\equiv 1 \pmod 3,\\
\delta^*(n)+1, & \text{if } n\equiv 2 \pmod 3.
\end{cases}
\]
If $n\equiv 1\pmod 3$, then~\eqref{eq:one-red-two-blue-sandwich} gives a
contradiction. Hence either $|R_1\cap Y|+|Y_{\out}|=\delta^*(n)$ or
$|R_1\cap Y|+|Y_{\out}|=\delta^*(n)+1$; in the latter case, we also have $n\equiv 2\pmod
3$.

First suppose that $|R_1\cap Y|+|Y_{\out}|=\delta^*(n)$. Then, for every
$x\in X_{\R} \cap R_1$, we have
\[
    \delta^*(n)\le d(x)\le |R_1\cap Y|+|Y_{\out}|=\delta^*(n).
\]
Thus $d(x)=|R_1\cap Y|+|Y_{\out}|$, and since
$N(x)\subseteq (R_1\cap Y)\cup Y_{\out}$, it follows that
\[
    N(x)=(R_1\cap Y)\cup Y_{\out}.
\]
In particular, every vertex of $Y_{\out}$ is adjacent to $x$, and all these
edges are blue, since a red edge between $x$ and a vertex in $Y_{\out}$ would
put this vertex in $R_1$. Therefore $Y_{\out}$ is contained in a single blue
component, say $B_3$. Since $Y_{\R}=\emptyset$, the components $B_1$, $B_2$ and
$B_3$ cover $Y$. By Proposition~\ref{prop:7}~\ref{prop:7-item2}, we obtain
$\tc_2(G,\varphi)\le 3$.

It remains to consider the case in which $n\equiv 2\pmod 3$ and
$|R_1\cap Y|+|Y_{\out}|=\delta^*(n)+1$. In this case,
\[
    2n-|R_1\cap Y|-|Y_{\out}|=2\delta^*(n).
\]
Therefore
\[
    2\delta^*(n)
    \le d(u)+d(v)
    \le (n-|R_1\cap Y|)+(n-|Y_{\out}|)
    =2\delta^*(n),
\]
and equality holds throughout. In particular,
\[
    d_R(u)+d_R(v)=n-|R_1\cap Y|.
\]
Since $u$ and $v$ have no common red neighbour and their red neighbours all lie
in $Y\setminus R_1$, we get
\[
    N_R(u)\cup N_R(v)=Y\setminus R_1.
\]
Hence the following three red components cover $Y$: $R_1$, the red component containing $u$, and the
red component containing $v$. By
Proposition~\ref{prop:7}~\ref{prop:7-item2}, we again get $\tc_2(G,\varphi)\le 3$.

This disposes of the case in which one of the two minimal covers has size at
most one. Therefore, in the remaining case, both minimal covers have size
exactly two.

\medskip
\noindent\textbf{Part 2:} $|\mathcal R|=2$ and $|\mathcal B|=2$.
\par\smallskip

Recall that $R_1$ and $R_2$ are the red components that
together cover the red-dominant vertices of $X$ and $B_1$ and $B_2$ are the
blue components that together cover the blue-dominant vertices of~$X$.

In Claims~\ref{claim:9},~\ref{claim:10} and~\ref{claim:11} we refine the sets $X_{\R}, X_{\B}, Y_{\R}, Y_{\B}$ to prove that we may reduce the covering problem in Figure~\ref{fig:structure_XR_XB} to the problem in
Figure~\ref{fig:structure_for_conclusion}, which has a significant number of
simplifying structural properties.

\begin{claim}\label{claim:9}
If at least one of $X_{\R}$, $X_{\B}$, $Y_{\R}$, or $Y_{\B}$ is empty, then
$\tc_2(G,\varphi)\le3$.
\end{claim}
\begin{claimproof}
If $X_{\R}=\emptyset$, then $B_1$ and $B_2$ cover $X$, and
Proposition~\ref{prop:5} gives $\tc_2(G,\varphi)\le3$. The case $X_{\B}=\emptyset$ is
analogous. Hence assume that $X_{\R}$ and $X_{\B}$ are non-empty.

By symmetry, it remains to prove the assertion when $Y_{\R}=\emptyset$. If
$Y_{\out}=\emptyset$, then $B_1$ and $B_2$ cover $Y$, and
Proposition~\ref{prop:5} again applies. We may therefore assume that
$Y_{\out}\neq\emptyset$.

First suppose that $X_{\R}\cap R_1\neq\emptyset$ and $X_{\R}\cap
R_2\neq\emptyset$. Consider arbitrary vertices $u\in X_{\R}\cap R_1$ and $v\in
X_{\R}\cap R_2$. At least one of the edges from any common neighbour $y$ to
$u$ and $v$ is blue, since $u$ and $v$ lie in distinct red components. Moreover,
$Y_{\R}=\emptyset$ implies $(R_1\cup R_2)\cap Y\subseteq B_1\cup B_2$. Hence a
common neighbour in any selected component would put $u$ or $v$ in
$B_1\cup B_2$, a contradiction. Thus $y\in Y_{\out}$, and neither $uy$ nor
$vy$ can be red. Both edges are therefore blue. Since $u$ and $v$ were
arbitrary, all vertices of $X_{\R}$ lie in a single blue component $B_3$.
Together with $B_1$ and $B_2$, this component covers $X$, so
Proposition~\ref{prop:7}~\ref{prop:7-item2} implies $\tc_2(G,\varphi)\le 3$.

Thus, we may assume that $X_{\R}\cap R_1\neq\emptyset$ and $X_{\R}\cap
R_2=\emptyset$. Let $u\in X_{\R}\cap R_1$ and note that the set of red
neighbours of $u$ lies in $R_1\cap Y$, and no blue neighbour of $u$ can lie in
$B_1\cup B_2$. Hence $N(u)\subseteq Y_{\out}\cup (R_1\cap Y)$, and thus
$|Y_{\out}|+|R_1\cap Y|\ge d(u)\ge \delta^*(n)$. Also, since $R_2\cap X$
contains a red-dominant vertex, $|R_2\cap Y|\ge \lceil\delta^*(n)/2\rceil$. It
follows that
\[
|Y_{\out}|+|R_1\cap Y|+|R_2\cap Y|
\ge \delta^*(n)+\left\lceil \frac{\delta^*(n)}{2}\right\rceil\ge n,
\]
which implies that $Y_{\B}=\emptyset$.

It remains to handle the vertices of $Y_{\out}$. If there is only one vertex $y$
in $Y_{\out}$, then, as every vertex has a red neighbour, $y$ belongs to some red
component, which together with $R_1$ and $R_2$ covers $Y$ (as
$Y_{\B}=\emptyset$) and, again, Proposition~\ref{prop:7}~\ref{prop:7-item2}
implies $\tc_2(G,\varphi)\le 3$.

Now suppose that $|Y_{\out}|\ge 2$. Let $u,v\in Y_{\out}$. Since
$\delta^*(n)>n/2$, the vertices $u$ and $v$ have a common neighbour. Moreover,
as $u$ and $v$ do not belong to any of $R_1$, $R_2$, $B_1$ and $B_2$, every
common neighbour of $u$ and $v$ lies in $X_{\R}\cup X_{\B}$. If such a common
neighbour lies in $X_{\R}$, then it is joined to both $u$ and $v$ by blue
edges; if it lies in $X_{\B}$, then it is joined to both $u$ and $v$ by red
edges. Thus every pair of vertices
of $Y_{\out}$ is joined by a monochromatic path of length $2$. By Fact~\ref{fact:auxiliary}, $Y_{\out}$ is contained in a single monochromatic component of $G$. If this
component is red, then together with $R_1$ and $R_2$ it gives three red
components covering $Y$; if it is blue, then together with $B_1$ and $B_2$ it
gives three blue components covering $Y$. In both cases,
Proposition~\ref{prop:7}~\ref{prop:7-item2} implies $\tc_2(G,\varphi)\le 3$.

\end{claimproof}

From Claim~\ref{claim:9}, we assume from now on that $X_{\R}, X_{\B}, Y_{\R}$
and $Y_{\B}$ are non-empty. The next claim identifies four further
configurations in which the desired cover already exists.

\begin{claim}\label{claim:10}
If one of the following holds, then $\tc_2(G,\varphi) \le 3$:
\begin{enumerate}[(i)]
\item $X_{\R}\cap R_1\neq\emptyset$ and $Y_{\R}\cap R_2=\emptyset$, or
      $Y_{\R}\cap R_1\neq\emptyset$ and $X_{\R}\cap R_2=\emptyset$;
\item $X_{\R}\cap R_2\neq\emptyset$ and $Y_{\R}\cap R_1=\emptyset$, or
      $Y_{\R}\cap R_2\neq\emptyset$ and $X_{\R}\cap R_1=\emptyset$;
\item $X_{\B}\cap B_1\neq\emptyset$ and $Y_{\B}\cap B_2=\emptyset$, or
      $Y_{\B}\cap B_1\neq\emptyset$ and $X_{\B}\cap B_2=\emptyset$;
\item $X_{\B}\cap B_2\neq\emptyset$ and $Y_{\B}\cap B_1=\emptyset$, or
      $Y_{\B}\cap B_2\neq\emptyset$ and $X_{\B}\cap B_1=\emptyset$.
\end{enumerate}
\end{claim}
\begin{claimproof}
By symmetry, it suffices to prove that (i) implies $\tc_2(G,\varphi) \le 3$. Suppose
that $X_{\R} \cap R_1 \neq \emptyset$ and $Y_{\R} \cap R_2 = \emptyset$ and
consider a vertex $u \in X_{\R} \cap R_1$. We have that $N(u) \subseteq (R_1 \cap Y) \cup Y_{\out}$, as there are no edges between $X_{\R}$ and $Y_{\B}$, so $|Y_{\out}| + |R_1
\cap Y| \ge \delta^*(n)$. Furthermore, we know there is a red-dominant vertex in
$R_2 \cap X$. Therefore, $|R_2 \cap Y| \ge \lceil \delta^*(n)/2 \rceil$, which,
by Fact~\ref{fact:simple}, implies
\begin{equation*}
|R_1 \cap Y|+|R_2 \cap Y|+|Y_{\out}|
\ge \delta^*(n)+\left\lceil\frac{\delta^*(n)}{2}\right\rceil\ge n.
\end{equation*}
Since $|Y|=n$, this means that $Y_{\B}=\emptyset$, so Claim~\ref{claim:9}
gives $\tc_2(G,\varphi)\le3$.

It remains to prove that $Y_{\R} \cap R_1 \neq \emptyset$ and $X_{\R} \cap R_2 =
\emptyset$ implies $\tc_2(G,\varphi) \le 3$. By our standing assumption, every vertex in
$G$ has neighbours in both colours. Therefore, given a vertex $u \in Y_{\R} \cap
R_1$, we know that $u$ has at least one blue neighbour. Since $u \in Y_{\R}$ and
$X_{\R}\cap R_2 = \emptyset$, the blue neighbours of $u$ belong to $X_{\R} \cap
R_1$. Therefore, $X_{\R} \cap R_1 \neq \emptyset$ and we already proved that in
this case, $Y_{\R} \cap R_2 = \emptyset$ implies $\tc_2(G,\varphi) \le 3$. Thus, we
assume that $Y_{\R} \cap R_2 \neq \emptyset$.

Now note that if $Y_{\out}\neq\emptyset$, then every
vertex of $Y_{\R}\cap R_2$ lies in the same blue component as every vertex of
$Y_{\out}$. Indeed, let $y\in Y_{\R}\cap R_2$ and $z\in Y_{\out}$. Since
$\delta^*(n)>n/2$, the vertices $y$ and $z$ have a common neighbour $x\in X$.

The common neighbour $x$ cannot lie in $(R_1\cup R_2)\cap(B_1\cup B_2)$.
Indeed, if $x\in R_i\cap B_j$ for some $i,j\in\{1,2\}$, then the edge $xz$ is
either red or blue; in the first case $z\in R_i$, and in the second case $z\in
B_j$, contradicting $z\in Y_{\out}$. Since $R_1\cup R_2\cup B_1\cup B_2$ covers
$X$, it follows that $x\in X_{\R}\cup X_{\B}$.

We claim that $x\notin X_{\B}$. If $xy$ were blue, then $y$ would lie in
$B_1\cup B_2$, contradicting $y\in Y_{\R}$. If $xy$ were red, then, since $y\in
R_2$, the vertex $x$ would lie in $R_2$, contradicting $x\in X_{\B}$. Hence
$x\in X_{\R}$. Since we are assuming $X_{\R}\cap R_2=\emptyset$, we have $x\in
X_{\R}\cap R_1$. Now $xy$ cannot be red, as this would merge $R_1$ and $R_2$,
and $xz$ cannot be red, as this would put $z$ in $R_1$. Hence both $xy$ and $xz$
are blue.

Now suppose that every pair of vertices $u\in Y_{\R}\cap R_1$ and $v\in
Y_{\R}\cap R_2$ has a common blue neighbour. Then all vertices of $Y_{\R}\cap
R_1$ and $Y_{\R}\cap R_2$ lie in a single blue component. Moreover, if
$Y_{\out}\neq\emptyset$, the previous paragraph shows that $Y_{\out}$ also lies
in this same blue component. Hence this blue component together with $B_1$ and
$B_2$ covers $Y$, and Proposition~\ref{prop:7}~\ref{prop:7-item2} implies that
$\tc_2(G,\varphi)\le 3$.

Thus, from now on we assume there exist vertices $u\in Y_{\R}\cap R_1$ and
$v\in Y_{\R}\cap R_2$ with no common blue neighbour. Then, from
$d_B(u)\geq 1$ and $d_B(v)\geq \delta^*(n)-d_R(v)$, we have

\begin{equation}
\label{claim:10-eq1}
|N_B(u) \cup N_B(v)| \ge 1 + (\delta^*(n) - d_R(v)).
\end{equation}
To obtain an upper bound for $d_R(v)$, note that $X_{\B} \neq \emptyset$ (from
Claim~\ref{claim:9}) and that all red neighbours of $v$ lie in $R_2\cap X$. Hence
\begin{align*}
d_R(v) \le |R_2 \cap X| &= n - |X_{\B}| - |R_1 \cap X|\\
& \le n - 1 - d(u)\\
& \le n - 1 - \delta^*(n).
\end{align*}
Here the equality uses the partition
$X=(R_1\cap X)\mathbin{\dot\cup}(R_2\cap X)\mathbin{\dot\cup}X_{\B}$, and the
second inequality follows from $|X_{\B}|\ge1$ and $N(u)\subseteq R_1\cap X$.

Substituting $d_R(v)\le n-1-\delta^*(n)$ into~\eqref{claim:10-eq1}, we obtain
\begin{equation*}
|N_B(u) \cup N_B(v)| \ge 2 + 2\delta^*(n) - n.
\end{equation*}
Since $\delta^*(n) \ge \frac{2n-1}{3}$, we have $|N_B(u) \cup N_B(v)| \ge n -
(\delta^*(n) - 1)$, which implies that the set $X \setminus (N_B(u) \cup
N_B(v))$ has at most $\delta^*(n) - 1$ vertices. Hence, since $\delta(G)\geq
\delta^*(n)$, any vertex of $Y$ has a neighbour in $N_B(u) \cup N_B(v)$. But by
definition, vertices of $Y_{\B}$ cannot have neighbours in $X_{\R}$, which
contains $N_B(u)\cup N_B(v)$. Hence $Y_{\B}=\emptyset$, and Claim~\ref{claim:9}
gives $\tc_2(G,\varphi)\le3$, completing the proof.
\end{claimproof}

The final claim further restricts how the four selected components meet the
red-only and blue-only sets.

\begin{claim}\label{claim:11}
If one of the following holds, then $\tc_2(G,\varphi) \le 3$:
\begin{enumerate}[(i)]
\item $X_{\R} \cap R_1 \neq \emptyset$ and $X_{\R} \cap R_2 \neq \emptyset$;
\item $X_{\B} \cap B_1 \neq \emptyset$ and $X_{\B} \cap B_2 \neq \emptyset$;
\item $Y_{\R} \cap R_1 \neq \emptyset$ and $Y_{\R} \cap R_2 \neq \emptyset$;
\item $Y_{\B} \cap B_1 \neq \emptyset$ and $Y_{\B} \cap B_2 \neq \emptyset$.
\end{enumerate}
\end{claim}
\begin{claimproof}
By Claim~\ref{claim:10}, in the remaining case item (i) can only occur together
with item (iii), and item (ii) can only occur together with item (iv). Thus, it
is enough to prove (i), as the proof of item (ii) follows analogously. Suppose
that $X_{\R} \cap R_1 \neq \emptyset$ and $X_{\R} \cap R_2 \neq \emptyset$. By
Claim~\ref{claim:10}, this implies $Y_{\R} \cap R_1 \neq \emptyset$ and $Y_{\R}
\cap R_2 \neq \emptyset$. Let $u \in Y_{\R} \cap R_1$ and $v \in Y_{\R} \cap
R_2$ and note that, since $2\delta^*(n) \le d(u) + d(v)$, we have
\begin{equation}\label{eq:2}
\begin{aligned}
2\delta^*(n) &\le (|R_1 \cap X| + d_B(u, R_2)) + (|R_2 \cap X| + d_B(v, R_1)) \\
&= n - |X_{\B}| + d_B(u, R_2) + d_B(v, R_1).
\end{aligned}
\end{equation}
Since $Y_{\B}\neq\emptyset$ by Claim~\ref{claim:9},
\[
d_B(u,R_2)+d_B(v,R_1)\le n-\delta^*(n).
\]
Indeed, both restricted blue neighbourhoods lie in $X_{\R}$, and the
definitions of $X_{\R}$ and $Y_{\B}$ forbid edges between these two sets. Together with~\eqref{eq:2},
this gives
\begin{equation*}
|X_{\B}| \le 2n - 3\delta^*(n) \le 2n - 3\left(\frac{2n-1}{3}\right) = 1.
\end{equation*}
Thus $|X_{\B}|=1$. Its unique vertex is incident to a red edge, so a red
component covers $X_{\B}$. This component together with $R_1$ and $R_2$ covers $X$, and
Proposition~\ref{prop:7}~\ref{prop:7-item2} gives $\tc_2(G,\varphi) \le 3$.
\end{claimproof}

After possibly interchanging $R_1$ and $R_2$ and, independently, $B_1$ and
$B_2$, Claims~\ref{claim:9}--\ref{claim:11} allow us to assume the following
configuration, shown in Figure~\ref{fig:structure_for_conclusion}:
\begin{equation}\label{eq:3}
\begin{aligned}
X_{\R}\cap R_1&\neq\emptyset, & Y_{\R}\cap R_2&\neq\emptyset,
& X_{\B}\cap B_1&\neq\emptyset, & Y_{\B}\cap B_2&\neq\emptyset,\\
X_{\R}\cap R_2&=\emptyset, & Y_{\R}\cap R_1&=\emptyset,
& X_{\B}\cap B_2&=\emptyset, & Y_{\B}\cap B_1&=\emptyset.
\end{aligned}
\end{equation}

If $Y_{\out}\neq\emptyset$, then, since $\delta^*(n)>n/2$, every pair
$u\in Y_{\R}\cap R_2$ and $v\in Y_{\out}$ has a common neighbour $x\in X$.
The vertex $x$ cannot lie in $(R_1\cup R_2)\cap(B_1\cup B_2)$, since the edge
$xv$ would put $v$ in one of the four selected components. Hence
$x\in X_{\R}\cup X_{\B}$. The case $x\in X_{\B}$ is impossible: either colour
on $xu$ would put $u$ in a selected blue component or $x$ in $R_2$. Thus,
by~\eqref{eq:3}, $x\in X_{\R}\cap R_1$, and both $xu$ and $xv$ are blue.
Consequently, $Y_{\out}\cup(Y_{\R}\cap R_2)$ lies in a blue component, which
together with $B_1$ and $B_2$ covers $Y$. Proposition~\ref{prop:7}~\ref{prop:7-item2}
then gives $\tc_2(G,\varphi)\le3$. We may therefore assume that
$Y_{\out}=\emptyset$ and introduce the following notation:
\[
\begin{aligned}
C_{11}&:=R_1\cap B_1\cap Y, & C_{21}&:=R_2\cap B_1\cap Y,
& S&:=R_2\cap Y_{\R},\\
C_{12}&:=R_1\cap B_2\cap Y, & C_{22}&:=R_2\cap B_2\cap Y,
& T&:=B_2\cap Y_{\B}.
\end{aligned}
\]
Because $Y_{\out}=\emptyset$ and~\eqref{eq:3} holds, the six sets
$C_{11},C_{12},C_{21},C_{22},S,T$ form a partition of $Y$.
By the minimality of $\mathcal R$ and $\mathcal B$, there is a red-dominant
vertex $u\in R_2\cap X$ and a blue-dominant vertex $v\in B_2\cap X$. Therefore,
\begin{align}
|C_{21}|+|C_{22}|+|S|
    &\ge d_R(u)\ge\left\lceil\delta^*(n)/2\right\rceil,
    \label{eq:delta*/2-1}\\
|C_{12}|+|C_{22}|+|T|
    &\ge d_B(v)\ge\left\lceil\delta^*(n)/2\right\rceil.
    \label{eq:delta*/2-2}
\end{align}

If $|X_{\R}\cap R_1|\le 1$, then the unique vertex of $X_{\R}\cap R_1$ has a
blue neighbour, necessarily in $Y_{\R}\cap R_2$, and therefore a blue component
together with $B_1$ and $B_2$ covers $X$. Hence
Proposition~\ref{prop:7}~\ref{prop:7-item2} gives $\tc_2(G,\varphi)\le 3$. Thus, we may
assume that $|X_{\R}\cap R_1| \geq 2$.

Note that if every pair of vertices in $X_{\R}\cap R_1$ has a common blue
neighbour (necessarily in $Y_{\R}\cap R_2$), then there would be three blue
components covering $X$, so Proposition~\ref{prop:7}~\ref{prop:7-item2} would
give $\tc_2(G,\varphi)\le3$. Thus, we may assume that there are vertices
$u_1$ and $u_2$ in $X_{\R} \cap R_1$ with no common blue neighbour. Analogously,
if $|X_{\B}\cap B_1|\le 1$, then the unique vertex of $X_{\B}\cap B_1$ has a red
neighbour, necessarily in $Y_{\B}\cap B_2$, and therefore a red component
together with $R_1$ and $R_2$ covers $X$. Hence
Proposition~\ref{prop:7}~\ref{prop:7-item2} gives $\tc_2(G,\varphi)\le 3$. We may
therefore assume that $|X_{\B}\cap B_1|\ge 2$ and, analogously, choose vertices
$v_1,v_2\in X_{\B}\cap B_1$ with no common red neighbour. For $i\in\{1,2\}$,
define
\begin{equation*}
S_i := N_B(u_i) \quad \text{and} \quad T_i := N_R(v_i).
\end{equation*}

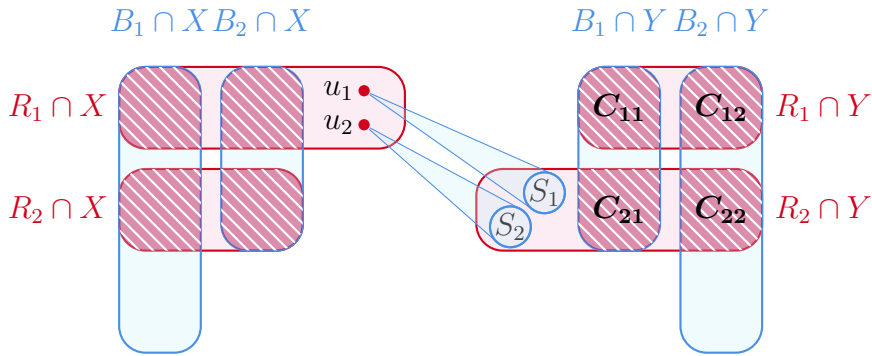
\begin{figure}[htbp]
    \centering
	\tikzset{every picture/.style={line width=0.75pt}}
\begin{tikzpicture}[scale=0.64,x=1pt,y=1pt]

\definecolor{myred}{rgb}{0.82, 0.01, 0.11}
\definecolor{myblue}{rgb}{0.29, 0.56, 0.89}
\definecolor{mypink}{rgb}{0.96, 0.76, 0.85}
\definecolor{mycyan}{rgb}{0.80, 0.95, 0.98}

\newcommand{\ConclusionBlueRedIntersection}[2]{%
  \begin{scope}
    \clip[rounded corners=10pt] #1;
    \clip[rounded corners=10pt] #2;
    \begin{pgfinterruptboundingbox}
      \fill[purple, opacity=0.4] (-20,-20) rectangle (200,200);
      \foreach \i in {-220,-212,...,180} {
        \draw[white, thick] (\i+220,-20) -- (\i,200);
      }
    \end{pgfinterruptboundingbox}
  \end{scope}%
}

\def\rmargin{6}
\draw[myred, thick, rounded corners=10pt, fill=mypink, fill opacity=0.3]
  (0+\rmargin,120+\rmargin) rectangle (180-\rmargin,180-\rmargin);

\draw[myred, thick, rounded corners=10pt, fill=mypink, fill opacity=0.3]
  (0+\rmargin,60+\rmargin) rectangle (120-\rmargin,120-\rmargin);

\def\bmargin{6}

\draw[myblue, thick, rounded corners=10pt, fill=mycyan, fill opacity=0.3]
  (0+\bmargin, 0+\bmargin) rectangle (60-\bmargin, 180-\bmargin);

\draw[myblue, thick, rounded corners=10pt, fill=mycyan, fill opacity=0.3]
  (60+\bmargin, 60+\bmargin) rectangle (120-\bmargin, 180-\bmargin);

\ConclusionBlueRedIntersection
  {(0+\rmargin,120+\rmargin) rectangle (180-\rmargin,180-\rmargin)}
  {(0+\bmargin,0+\bmargin) rectangle (60-\bmargin,180-\bmargin)}
\ConclusionBlueRedIntersection
  {(0+\rmargin,120+\rmargin) rectangle (180-\rmargin,180-\rmargin)}
  {(60+\bmargin,60+\bmargin) rectangle (120-\bmargin,180-\bmargin)}
\ConclusionBlueRedIntersection
  {(0+\rmargin,60+\rmargin) rectangle (120-\rmargin,120-\rmargin)}
  {(0+\bmargin,0+\bmargin) rectangle (60-\bmargin,180-\bmargin)}
\ConclusionBlueRedIntersection
  {(0+\rmargin,60+\rmargin) rectangle (120-\rmargin,120-\rmargin)}
  {(60+\bmargin,60+\bmargin) rectangle (120-\bmargin,180-\bmargin)}

\node[myred] at (-30, 150) {$R_1 \cap X$};
\node[myred] at (-30, 90)  {$R_2 \cap X$};
\node[myblue] at (30, 200) {$B_1 \cap X$};
\node[myblue] at (90, 200)  {$B_2 \cap X$};

\begin{scope}[shift={(210, 0)}]

	\def\rmargin{6}
	\draw[myred, thick, rounded corners=10pt, fill=mypink, fill opacity=0.3]
	  (60+\rmargin,120+\rmargin) rectangle (180-\rmargin,180-\rmargin);

	\draw[myred, thick, rounded corners=10pt, fill=mypink, fill opacity=0.3]
	  (0+\rmargin,60+\rmargin) rectangle (180-\rmargin,120-\rmargin);

	\def\bmargin{6}

	\draw[myblue, thick, rounded corners=10pt, fill=mycyan, fill opacity=0.3]
	  (60+\bmargin, 60+\bmargin) rectangle (120-\bmargin, 180-\bmargin);

	\draw[myblue, thick, rounded corners=10pt, fill=mycyan, fill opacity=0.3]
	  (120+\bmargin, 0+\bmargin) rectangle (180-\bmargin, 180-\bmargin);

	\ConclusionBlueRedIntersection
	  {(60+\rmargin,120+\rmargin) rectangle (180-\rmargin,180-\rmargin)}
	  {(60+\bmargin,60+\bmargin) rectangle (120-\bmargin,180-\bmargin)}
	\ConclusionBlueRedIntersection
	  {(60+\rmargin,120+\rmargin) rectangle (180-\rmargin,180-\rmargin)}
	  {(120+\bmargin,0+\bmargin) rectangle (180-\bmargin,180-\bmargin)}
	\ConclusionBlueRedIntersection
	  {(0+\rmargin,60+\rmargin) rectangle (180-\rmargin,120-\rmargin)}
	  {(60+\bmargin,60+\bmargin) rectangle (120-\bmargin,180-\bmargin)}
	\ConclusionBlueRedIntersection
	  {(0+\rmargin,60+\rmargin) rectangle (180-\rmargin,120-\rmargin)}
	  {(120+\bmargin,0+\bmargin) rectangle (180-\bmargin,180-\bmargin)}

	\node[myred] at (210, 150) {$R_1 \cap Y$};
	\node[myred] at (210, 90)  {$R_2 \cap Y$};
	\node[myblue] at (90, 200) {$B_1 \cap Y$};
	\node[myblue] at (150, 200) {$B_2 \cap Y$};

	\node[font=\boldmath] at (90, 150) {$C_{11}$};
	\node[font=\boldmath] at (150, 150) {$C_{12}$};
	\node[font=\boldmath] at (90, 90) {$C_{21}$};
	\node[font=\boldmath] at (150, 90) {$C_{22}$};
\end{scope}

\node[circle, fill=myred, inner sep=1.5pt] at (150, 160) (u_1) {};
\node[circle, fill=myred, inner sep=1.5pt] at (150, 140) (u_2) {};

\node[] at (135, 160) {$u_1$};
\node[] at (135, 140) {$u_2$};

\begin{scope}[shift={(210,0)}]
	\node[] at (40+\bmargin,100) (S_1) {$S_1$};
	\node[] at (20+\bmargin,80) (S_2) {$S_2$};
	\draw[myblue, thick] (S_1) circle (12);
	\draw[myblue, thick] (S_2) ellipse (12);
\end{scope}

\coordinate (T1a) at (249.09,90.19);
\coordinate (T1b) at (260.85,110.97);
\coordinate (T2a) at (228.05,71.01);
\coordinate (T2b) at (241.69,90.56);

\fill[mycyan, fill opacity=0.3] (u_1) -- (T1a) arc (234.835:66.142+360:12) -- cycle;
\draw[myblue, thin] (u_1) -- (T1a);
\draw[myblue, thin] (u_1) -- (T1b);

\fill[mycyan, fill opacity=0.3] (u_2) -- (T2a) arc (228.526:61.669+360:12) -- cycle;
\draw[myblue, thin] (u_2) -- (T2a);
\draw[myblue, thin] (u_2) -- (T2b);

\draw[myblue, thick] (S_1) circle (12);
\draw[myblue, thick] (S_2) ellipse (12);

\end{tikzpicture}
    \caption{Disjoint blue neighbourhoods of two vertices in $X_{\R}$. The
    symmetric red neighbourhoods $T_1,T_2\subseteq T$ are omitted.}
    \label{fig:structure_for_conclusion}
\end{figure}

By the current structure, $S_i\subseteq S$ and $T_i\subseteq T$ for $i=1,2$.
Moreover, $S_1\cap S_2=\emptyset$ and $T_1\cap T_2=\emptyset$.

The minimum-degree condition gives
\begin{align}
|S_1|+|C_{11}|+|C_{12}| &\ge d(u_1)\ge\delta^*(n), \label{eq:delta*-1}\\
|S_2|+|C_{11}|+|C_{12}| &\ge d(u_2)\ge\delta^*(n), \label{eq:delta*-2}\\
|T_1|+|C_{11}|+|C_{21}| &\ge d(v_1)\ge\delta^*(n), \label{eq:delta*-3}\\
|T_2|+|C_{11}|+|C_{21}| &\ge d(v_2)\ge\delta^*(n). \label{eq:delta*-4}
\end{align}

It remains to show that this forced configuration is impossible.

Combining~\eqref{eq:delta*/2-1},~\eqref{eq:delta*/2-2}, and
\eqref{eq:delta*-1}--\eqref{eq:delta*-4}, and using
$|S|\ge|S_1|+|S_2|$ and $|T|\ge|T_1|+|T_2|$, we obtain
\begin{equation*}
2\underbrace{\bigl(|S|+|T|+|C_{11}|+|C_{12}|+|C_{21}|+|C_{22}|\bigr)}_{n}
+2|C_{11}|+|C_{12}|+|C_{21}|
\ge 4\delta^*(n)+2\left\lceil\frac{\delta^*(n)}{2}\right\rceil.
\end{equation*}
Rearranging and using Fact~\ref{fact:simple}, we obtain
\begin{equation*}
(|C_{11}| + |C_{12}|) + (|C_{11}| + |C_{21}|) \ge 2\delta^*(n).
\end{equation*}
Hence $|C_{11}| + |C_{12}| \ge \delta^*(n)$ or $|C_{11}| + |C_{21}| \ge
\delta^*(n)$. If $|C_{11}| + |C_{12}| \ge \delta^*(n)$, then together
with~\eqref{eq:delta*/2-1} and Fact~\ref{fact:simple} we get
\begin{equation*}
|C_{11}| + |C_{12}| + |C_{21}| + |C_{22}| + |S| \ge \delta^*(n) + \left\lceil \frac{\delta^*(n)}{2} \right\rceil \ge n.
\end{equation*}
Thus $T=\emptyset$, contradicting~\eqref{eq:3}. If instead
$|C_{11}|+|C_{21}|\ge \delta^*(n)$, then the analogous argument
using~\eqref{eq:delta*/2-2}
gives $S=\emptyset$, again contradicting~\eqref{eq:3}. Hence
$\tc_2(G,\varphi) \le 3$. Since $\varphi$ was arbitrary, we conclude that
$\tc_2(G) \le 3$.
\end{proof}

\section*{Acknowledgements}
We would like to thank Matías Pavez-Signé for inviting us to the Santiago Summer
Workshop in Combinatorics, where this work was started. C. Bispo was supported
by CAPES (88887.015387/2024-00). G. Kontogeorgiou was supported by ANID (Grant
CMM Basal FB210005) and ANID-FONDECYT Postdoctorado (Grant 3250479). M. Lage was
supported by FAPESP (2025/06707-6 and 2025/14743-2). G. O. Mota was supported by
CNPq (420838/2025-2 and 315916/2023-0) and FAPESP (2024/13859-4 and
2023/03167-5). B. Skarmeta was supported by ANID (Grant CMM Basal FB210005) and Fondecyt Regular (Grant 1241398). This study was financed in
part by CAPES, Coordenação de Aperfeiçoamento de Pessoal de Nível Superior,
Brasil, Finance Code~001.

\bibliographystyle{plain}
\bibliography{refs}

\end{document}